\documentclass[12pt]{article}

\usepackage{authblk}
\usepackage{amsmath}
\usepackage{graphicx}
\usepackage{enumerate}
\usepackage{natbib}
\usepackage{url}
\usepackage{tcolorbox}

\usepackage{enumitem}
\usepackage{amsmath}    
\usepackage{amssymb}    
\usepackage{amsthm}     

\theoremstyle{plain}
\newtheorem{theorem}{Theorem}
\newtheorem{proposition}[theorem]{Proposition}

\newtheorem{corollary}[theorem]{Corollary}

\theoremstyle{definition}

\newtheorem{example}{Example}[section]

\newtheorem{remark}[theorem]{Remark}

\usepackage{caption}
\usepackage{graphicx}
\usepackage{xcolor}
\usepackage{algorithm}
\usepackage{algorithmic}
\usepackage{datetime2} 

\usepackage{tikz}
\usetikzlibrary{positioning,arrows.meta}
\usepackage{booktabs}  

\begin{document}

\setlength{\parskip}{0.4\baselineskip plus 2pt}
\setlength{\parindent}{0pt}

\def\spacingset#1{\renewcommand{\baselinestretch}%
{#1}\small\normalsize} \spacingset{1}

\title{\bf Intrinsic and Normal Mean Ricci Curvatures:\par
  A Bochner–Weitzenb\"ock Identity for Simple $d$-Vectors\thanks{Research reported in this publication was supported by grant number INV-048956 from the Gates Foundation.}}

\author{%
  Pawe\l{} Gajer$^{1}$\thanks{Corresponding author: pgajer@gmail.com}\; and Jacques Ravel$^{1}$%
}

\date{} 

\footnotetext[1]{Center for Advanced Microbiome Research and Innovation (CAMRI), Institute for Genome Sciences, and Department of Microbiology and Immunology, University of Maryland School of Medicine}

\maketitle

\begin{abstract}
We present a concise, coordinate‐free framework that packages pointwise curvature information into two elementary subspace averages attached to a $d$–plane $\Pi\subset T_pM$: the \emph{intrinsic mean Ricci}
\[
\overline{\mathrm{Ric}}_{\Pi}=\frac{2}{d}\sum_{1\le i<j\le d}K(e_i,e_j),
\]
and the \emph{normal (mixed) mean Ricci}
\[
\overline{\mathrm{Ric}}^{\perp}_{\Pi}=\frac{1}{d(n-d)}\sum_{i=1}^{d}\sum_{\alpha=1}^{n-d}K(e_i,n_\alpha).
\]
Via Jacobi–field expansions, these quantities appear as the $r^2/6$ coefficients
in the volume elements of (i) the intrinsic $(d\!-\!1)$–sphere inside $\Pi$ and
(ii) the normal $(n\!-\!d\!-\!1)$–sphere in $\Pi^\perp$, unifying classical
small‐sphere/ball and tube formulas. A key payoff is a “plug-and-play”
Bochner–Weitzenb\"ock identity for \emph{simple} $d$–vectors: if
$V=X_1\wedge\cdots\wedge X_d$ is orthonormal and $\Pi=\mathrm{span}\{X_i\}$,
then
\[
\langle \mathcal R_d V,V\rangle
= d(n-d)\,\overline{\mathrm{Ric}}^{\perp}_{\Pi}.
\]
This yields immediate analytic consequences: a Bochner vanishing criterion for
harmonic simple $d$–vectors under
$\inf_{\Pi}\overline{\mathrm{Ric}}^{\perp}_{\Pi}>0$, and a Lichnerowicz-type
eigenvalue lower bound
$\lambda\ge d(n-d)\inf_{\Pi}\overline{\mathrm{Ric}}^{\perp}_{\Pi}$ for simple
eigenfields.
\end{abstract}

\newpage

Scalar, sectional, and Ricci curvatures all appear as the quadratic defects in
Euclidean volume growth. In geodesic polar coordinates around $p\in(M^n,g)$ the
geodesic sphere $S_r(p)$ has $(n{-}1)$-dimensional volume
\begin{equation}\label{eq:geosphere}
\mathrm{vol}(S_{r}(p))=\omega_{n-1}\,r^{n-1}\Big(1-\tfrac{\mathrm{Scal}(p)}{6n}\,r^{2}+O(r^{4})\Big),
\end{equation}
where $\mathrm{Scal}(p)$ is the scalar curvature at $p$; see, e.g.,
\cite[Ch.~2--3]{gray2004tubes}, \cite[§6.1--6.2]{petersen2016riemannian},
\cite[Ch.~1--2]{cheeger1975comparison}. For a unit vector $u\in T_pM$ and the
geodesic $\gamma_u(t)=\exp_p(tu)$, the normal $(n{-}2)$-sphere bundle along
$\gamma_u$ has fiber volume element
\begin{equation}\label{eq:normal_sphere} dV_{S^{\,n-2}_{r,u^\perp}}=r^{n-2}\Big(1-\tfrac{\mathrm{Ric}(u,u)}{6}\,r^{2}+O(r^{4})\Big)\,dV_{S^{n-2}},
\end{equation}
where $\mathrm{Ric}(u,u)$ is the Ricci curvature in the direction $u$
\cite[Ch.~2--3]{gray2004tubes}, \cite[§6.1]{petersen2016riemannian}. If
$\Pi\subset T_pM$ is a two-dimensional subspace, then for small $r>0$ the
geodesic disk in $\Pi$ of radius $r$ has boundary length
\begin{equation}\label{eq:BDP}
L_{\Pi}(r)=2\pi r\Big(1-\tfrac{K(\Pi)}{6}\,r^{2}+O(r^{4})\Big),
\end{equation}
the classical Bertrand--Diguet--Puiseux formula, which measures the intrinsic
curvature of the geodesic disk $\exp_p(\Pi)$; see
\cite[§4--5]{docarmo1992riemannian}.

These examples illustrate two averaging patterns that generalize to all dimensions.

\medskip
\noindent\textbf{The intrinsic pattern.}
Fix a $d$-dimensional subspace $\Pi\subset T_pM$ and average sectional curvature
over all $2$-planes contained in $\Pi$. For $2\le d\le n$, the $(d{-}1)$-sphere
in $\Pi$ of radius $r$ has volume element \cite[Ch.~2--3]{gray2004tubes}:
\begin{equation}\label{eq:pi_sphere}
dV_{S^{\,d-1}_{r,\Pi}}=r^{d-1}\Big(1-\tfrac{r^{2}}{6}\,\mathrm{Ric}_{\Pi}(v,v)+O(r^{4})\Big)\,dV_{S^{d-1}_\Pi},
\end{equation}
where, for $v\in S^{d-1}_\Pi$ and any orthonormal basis $\{v,e_1,\ldots,e_{d-1}\}$ of $\Pi$,
\[
\mathrm{Ric}_{\Pi}(v,v)=\sum_{i=1}^{d-1}K(v,e_{i}).
\]
Averaging over $v$ gives the mean intrinsic Ricci curvature
\begin{equation}\label{eq:mean-ricci-defs}
\overline{\mathrm{Ric}}_{\Pi}
=\frac{1}{\omega_{d-1}}\!\int_{S^{d-1}_\Pi}\!\!\mathrm{Ric}_{\Pi}(v,v)\,dV_{S^{d-1}_\Pi}(v)
=\frac{2}{d}\sum_{1\le i<j\le d}K(e_{i},e_{j}),
\end{equation}
which reduces to $K(\Pi)$ when $d=2$ and to $\mathrm{Scal}(p)/n$ when $d=n$.

\noindent\textbf{The perpendicular pattern.}
Fix $\Pi$ and average over all $2$-planes spanned by one vector in $\Pi$ and one in $\Pi^{\perp}$. In analogy with $\mathrm{Ric}_\Pi(v,v)$, define for $u\in S^{d-1}_\Pi$ the \emph{normal} (mixed) Ricci
\[
\mathrm{Ric}^{\perp}_{\Pi}(u)=\sum_{\alpha=1}^{n-d}K(u,n_{\alpha}),
\]
where $\{n_1,\ldots,n_{n-d}\}$ is an orthonormal basis of $\Pi^\perp$. The normal $(n{-}d{-}1)$-sphere in $\Pi^\perp$ of radius $r$ has volume element
\begin{equation}\label{eq:normal_sphere_pi}
dV_{S^{\,n-d-1}_{r,\Pi^\perp}}=r^{n-d-1}\Big(1-\tfrac{r^{2}}{6}\,\mathrm{Ric}^{\perp}_{\Pi}(u)+O(r^{4})\Big)\,dV_{S^{n-d-1}},
\end{equation}
and averaging over $u$ defines the mean normal Ricci curvature
\begin{equation}\label{eq:mean-normal-ricci-def}
\overline{\mathrm{Ric}}^{\perp}_{\Pi}
=\frac{1}{\omega_{d-1}}\!\int_{S^{d-1}_\Pi}\!\!\mathrm{Ric}^{\perp}_{\Pi}(u)\,dV_{S^{d-1}_\Pi}(u)
=\frac{1}{d(n-d)}\sum_{i=1}^{d}\sum_{\alpha=1}^{n-d}K(e_{i},n_{\alpha}).
\end{equation}
This perpendicular mean is defined purely from the splitting
$T_pM=\Pi\oplus\Pi^{\perp}$ and does not require $\Pi$ to be tangent to a
submanifold. For $d=2$, $\overline{\mathrm{Ric}}_\Pi=K(\Pi)$, for $d=1$,
$\overline{\mathrm{Ric}}^{\perp}_{\Pi}=\mathrm{Ric}(u,u) / (n-1)$, and for $d=n$,
$\overline{\mathrm{Ric}}_\Pi=\mathrm{Scal}/n$.

Thus $\overline{\mathrm{Ric}}_{\Pi}$ is the \emph{intrinsic} mean of sectional curvatures within $\Pi$, while $\overline{\mathrm{Ric}}^{\perp}_{\Pi}$ is the \emph{normal (mixed)} mean across directions orthogonal to $\Pi$.
\begin{center}
  \includegraphics[scale=0.6]{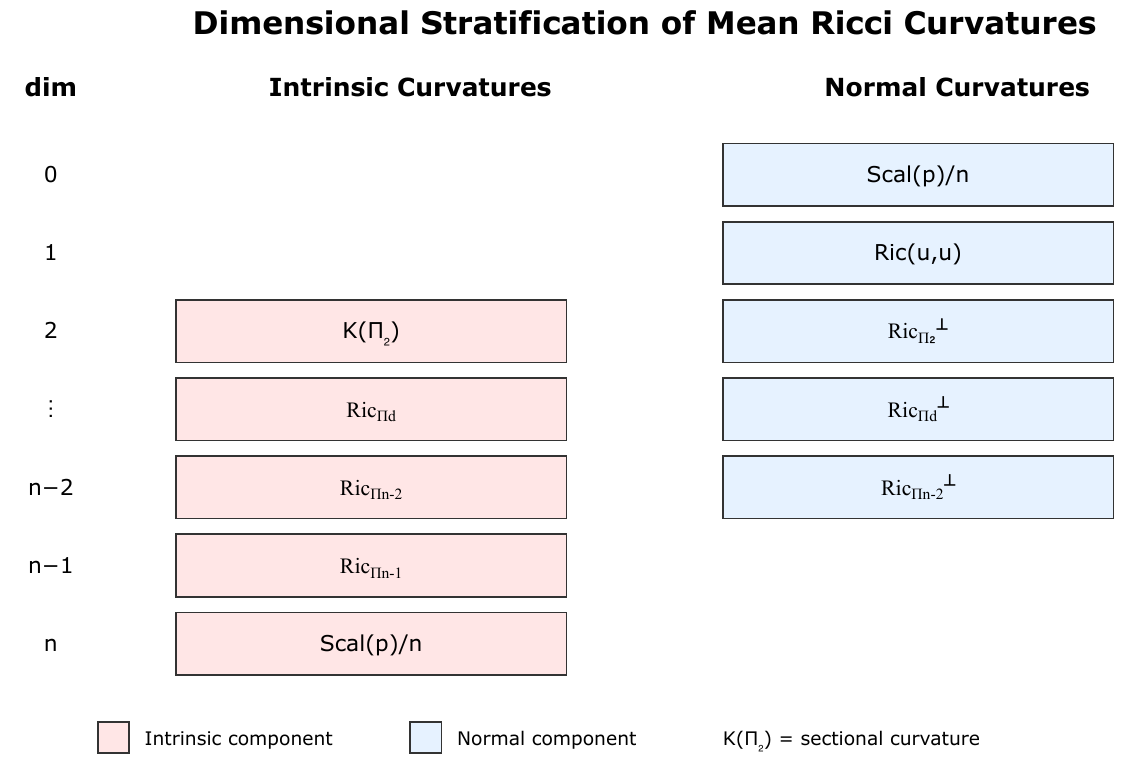}
  \captionof{figure}{Dimensional stratification of mean Ricci curvatures for
    linear subspaces $\Pi_d \subset T_pM$ of dimension $d$ in an $n$-dimensional
    Riemannian manifold. The intrinsic curvature $\overline{\text{Ric}}_{\Pi_d}$
    measures the mean of sectional curvatures within $\Pi_d$, while the normal
    curvature $\overline{\text{Ric}}_{\Pi_d}^\perp$ measures the mean of
    sectional curvatures for planes containing one direction in $\Pi_d$ and one
    in $\Pi_d^\perp$. Classical curvatures emerge as special cases: scalar
    curvature appears at dimensions 0 and $n$, Ricci curvature at dimension 1,
    and sectional curvature at dimension 2. The complementary pattern reflects
    the orthogonal decomposition
    $T_pM = \Pi_d \oplus \Pi_d^\perp$. \label{fig1}}
\end{center}
\noindent\textbf{Bochner--Weitzenb\"ock for simple $p$-vectors.}
Let
\[
\Delta_{H} = \nabla^*\nabla+\mathcal R_p
\]
denote the Hodge/Lichnerowicz Laplacian acting on $\Lambda^pTM$ (via the musical
isomorphism). For any bivector field $B$,
\[
\frac{1}{2}\,\Delta |B|^2
=\langle \nabla^*\nabla B,B\rangle - |\nabla B|^2 + \langle \mathcal R_2 B,B\rangle,
\]
where $\Delta$ on functions denotes the nonnegative Laplacian
$-\operatorname{div}\nabla$. If $B$ is \emph{simple and unit}, $B=X\wedge Y$
with $X\perp Y$ and $\Pi=\mathrm{span}\{X,Y\}$, then
\begin{equation}\label{eq:bochner-bivector}
\big\langle \mathcal R_2(X\wedge Y),\,X\wedge Y\big\rangle
=\mathrm{Ric}(X,X)+\mathrm{Ric}(Y,Y)-2K(\Pi)
=2(n-2)\,\overline{\mathrm{Ric}}^{\perp}_{\Pi},
\end{equation}
so
\[
\frac{1}{2}\,\Delta |X\wedge Y|^2
=\langle \nabla^*\nabla (X\wedge Y),\,X\wedge Y\rangle - |\nabla (X\wedge Y)|^2
+ 2(n-2)\,\overline{\mathrm{Ric}}^{\perp}_{\Pi}\,|X\wedge Y|^2,
\]
without unpacking the full curvature operator on $\Lambda^2$; see
\cite[§§2.1--2.3]{rosenberg1997laplacian}, \cite[§7.4]{petersen2016riemannian},
\cite{bochner1946curvature}. More generally, for any simple orthonormal
$d$-vector $V=X_1\wedge\cdots\wedge X_d$ spanning
$\Pi=\mathrm{span}\{X_1,\dots,X_d\}$ ($2\le d\le n-1$), the Weitzenböck
curvature term depends only on the mixed sectional curvatures:

\begin{proposition}[Bochner curvature term for simple $d$-vectors]\label{prop:Rd-simple}
  Let $(M^n,g)$ be a closed Riemannian manifold and fix $p\in M$. Let
  $\{X_1,\dots,X_d\}$ be an orthonormal basis of $\Pi\subset T_pM$, let
  $\{X_1,\dots,X_d,n_1,\dots,n_{n-d}\}$ be an orthonormal basis of $T_pM$, and
  let $V=X_1\wedge\cdots\wedge X_d\in\Lambda^dT_pM$. With
  $R(X,Y)Z=\nabla_X\nabla_YZ-\nabla_Y\nabla_XZ-\nabla_{[X,Y]}Z$, one has
\[
\big\langle \mathcal R_d V,\,V\big\rangle
=\sum_{i=1}^{d}\sum_{\alpha=1}^{n-d} K(X_i,n_\alpha)
=d(n-d)\,\overline{\mathrm{Ric}}^{\perp}_{\Pi}.
\label{eq:Rd-on-simple}
\]
Consequently,
\[
\frac12\,\Delta|V|^2
=\langle \nabla^*\nabla V,\,V\rangle - |\nabla V|^2
+ d(n-d)\,\overline{\mathrm{Ric}}^{\perp}_{\Pi}\,|V|^2 .
\label{eq:bochner-simple-d}
\]
\end{proposition}

\noindent In particular, for $d=2$ this recovers \eqref{eq:bochner-bivector}. See \cite[§1.C]{Besse} and \cite[§7.4]{petersen2016riemannian} for background on the Weitzenböck curvature endomorphism on $\Lambda^d$.

\medskip
\noindent\textbf{Immediate applications.}
Let $2\le d\le n-1$ and set
\[
  \kappa_d \;:=\; d(n-d)\,\inf_{\Pi\in G_d(T_pM),\,p\in M}\,\overline{\mathrm{Ric}}^{\perp}_{\Pi}.
\]

\begin{corollary}[Bochner vanishing for simple $d$-vectors]\label{cor:vanishing-simple}
  If $\kappa_d>0$ on a closed Riemannian manifold $(M^n,g)$, then there are no
  nonzero harmonic simple $d$-vector fields.
\end{corollary}

\begin{proof}
Let $V$ be a smooth simple $d$-vector field on a closed Riemannian manifold $(M^n,g)$ with $\Delta_H V=0$, where
\[
\Delta_H V=\nabla^*\nabla V+\mathcal R_d V
\]
is the Hodge/Lichnerowicz Laplacian on $\Lambda^d TM$. Take the $L^2$ inner product with $V$ and integrate:
\[
0=\int_M \langle \Delta_H V, V\rangle
 = \int_M \langle \nabla^*\nabla V, V\rangle + \int_M \langle \mathcal R_d V, V\rangle.
\]
By adjointness on a closed manifold,
\[
\int_M \langle \nabla^*\nabla V, V\rangle = \int_M |\nabla V|^2 = \|\nabla V\|_{L^2}^2,
\]
hence
\[
0 = \|\nabla V\|_{L^2}^2 + \int_M \langle \mathcal R_d V, V\rangle.
\tag{$\ast$}
\]
If $V$ is pointwise simple, write $V=|V|\,U$ with $U$ unit simple. Then Proposition~\ref{prop:Rd-simple} yields
\[
\langle \mathcal R_d V, V\rangle
= |V|^2\,\langle \mathcal R_d U, U\rangle
= d(n-d)\,\overline{\mathrm{Ric}}^{\perp}_{\Pi_U}\,|V|^2
\;\ge\; \kappa_d\,|V|^2,
\]
where $\displaystyle \kappa_d:= d(n-d)\,\inf_{p\in M}\ \inf_{\Pi\in G_d(T_pM)} \overline{\mathrm{Ric}}^{\perp}_{\Pi}$.
Plugging this into $(\ast)$ gives
\[
0 \;\ge\; \|\nabla V\|_{L^2}^2 + \kappa_d \|V\|_{L^2}^2.
\]
If $\kappa_d>0$ then both terms must vanish, so $V\equiv 0$.
\end{proof}

\begin{proposition}[Lichnerowicz–type lower bound for simple $d$-vector eigenfields]\label{prop:lichnerowicz-simple}
  Let $(M^n,g)$ be closed Riemannian manifold and let
  $$
  \kappa_d:=d(n-d)\,\inf_{p\in M}\ \inf_{\Pi\in G_d(T_pM)}\overline{\mathrm{Ric}}^{\perp}_{\Pi}.
  $$
  If $V\not\equiv 0$ is a smooth \emph{pointwise simple} $d$-vector field satisfying
  \[
    \Delta_H V=\lambda V \qquad\text{on }\Lambda^d TM,
  \]
  then
  \[
    \lambda \;\ge\; \kappa_d.
  \]
  Equivalently, for every smooth pointwise simple $d$-vector field $V$,
  \[
    \frac{\displaystyle \int_M \langle \Delta_H V,\,V\rangle}{\displaystyle \int_M |V|^2}
    \;\ge\; \kappa_d.
  \]
\end{proposition}

\begin{proof}
Pair $\Delta_H V=\lambda V$ with $V$ and integrate:
\[
\lambda \|V\|_{L^2}^2 \;=\; \int_M \langle \Delta_H V, V\rangle
= \int_M \langle \nabla^*\nabla V, V\rangle + \int_M \langle \mathcal R_d V, V\rangle
= \|\nabla V\|_{L^2}^2 + \int_M \langle \mathcal R_d V, V\rangle,
\]
using adjointness on a closed manifold. If $V=|V|\,U$ with $U$ unit and simple pointwise, Proposition~\ref{prop:Rd-simple} gives
\[
\langle \mathcal R_d V, V\rangle = |V|^2\,\langle \mathcal R_d U, U\rangle
= d(n-d)\,\overline{\mathrm{Ric}}^{\perp}_{\Pi_U}\,|V|^2 \;\ge\; \kappa_d\,|V|^2,
\]
hence
\[
\lambda \|V\|_{L^2}^2 \;\ge\; \|\nabla V\|_{L^2}^2 + \kappa_d \|V\|_{L^2}^2 \;\ge\; \kappa_d \|V\|_{L^2}^2.
\]
Dividing by \(\|V\|_{L^2}^2\) yields \(\lambda\ge \kappa_d\).
\end{proof}

\begin{remark}
If equality $\lambda=\kappa_d$ holds, then necessarily $\nabla V\equiv 0$ and
$\langle \mathcal R_d U,U\rangle \equiv \kappa_d$ for the unit simple field $U=V/|V|$ wherever $V\neq 0$.
\end{remark}

\begin{remark}
  These statements are \emph{restricted to simple fields}. A global lower bound
  for the full Hodge Laplacian on $\Lambda^d$ requires a pointwise operator
  inequality $\langle \mathcal R_d w,w\rangle \ge \kappa\,|w|^2$ for all
  $w\in\Lambda^d$, which is stronger than positivity of the mean normal Ricci on
  planes.
\end{remark}

\begin{remark}[Dictionary: curvature-operator traces]
Let $\mathcal R:\Lambda^2 T_pM\to\Lambda^2 T_pM$ be the curvature operator determined by
\[
\langle \mathcal R(u\wedge v),\,w\wedge z\rangle \;=\; R(u,v,w,z).
\]
If $\Pi\subset T_pM$ has $\dim\Pi=d$ and orthonormal bases $\{e_1,\dots,e_d\}$ of $\Pi$ and $\{n_1,\dots,n_{n-d}\}$ of $\Pi^\perp$, then
\[
\operatorname{tr}\!\left(\mathcal R\big|_{\Lambda^2\Pi}\right)
=\sum_{1\le i<j\le d} K(e_i,e_j),\qquad
\operatorname{tr}\!\left(\mathcal R\big|_{\Pi\wedge \Pi^\perp}\right)
=\sum_{i=1}^{d}\sum_{\alpha=1}^{n-d} K(e_i,n_\alpha).
\]
Hence
\[
\overline{\mathrm{Ric}}_{\Pi}=\frac{2}{d}\,\operatorname{tr}\!\left(\mathcal R\big|_{\Lambda^2\Pi}\right),\qquad
\overline{\mathrm{Ric}}^{\perp}_{\Pi}=\frac{1}{d(n-d)}\,\operatorname{tr}\!\left(\mathcal R\big|_{\Pi\wedge \Pi^\perp}\right),
\]
which are basis-independent and depend only on the splitting $T_pM=\Pi\oplus\Pi^\perp$.
\end{remark}

\paragraph{Tube coefficients and the averaged curvatures.}
As a companion to the small-sphere densities \eqref{eq:pi_sphere}–\eqref{eq:normal_sphere_pi}, the same Jacobi–field/Jacobian expansions enter classical tube–volume formulas. If $P^d\subset M^n$ is an embedded submanifold with tangent spaces $\Pi=T_pP$ and normal bundle $\nu(P)$, then for $0<r<\operatorname{inj}_\nu(P)$ the volume of the tube of radius $r$ about $P$ has the expansion
\[
\mathrm{vol}(\mathrm{Tube}_{r}(P))=\sum_{j=0}^{\infty} c_{j}(P)\, r^{\,n-d+j}.
\]
The leading term $c_0(P)=\omega_{n-d}\,\mathrm{vol}(P)$ is Euclidean, and $c_1(P)$ is linear in the mean curvature (vanishing for symmetric two–sided tubes or for minimal $P$). A local Jacobian computation along normal geodesics (cf.\ \cite{weyl1939volume,gray1990tubes,gray2004tubes}) shows that the first curvature–dependent correction is
\begin{equation}\label{eq:c2decomp}
c_{2}(P)
= -\frac{\omega_{n-d}}{6}\!\int_{P}\!\big[\,\alpha_{n,d}\,\overline{\mathrm{Ric}}_{\Pi}
+\beta_{n,d}\,\overline{\mathrm{Ric}}^{\perp}_{\Pi}
+\gamma_{n,d}\,\|A\|^{2}
+\delta_{n,d}\,\|H\|^{2}\,\big]\;dV_{P},
\end{equation}
where $A$ is the second fundamental form and $H$ the mean–curvature vector; the constants $\alpha_{n,d},\beta_{n,d},\gamma_{n,d},\delta_{n,d}$ depend only on $(n,d)$ (see \cite{gray2004tubes}). In the totally geodesic case ($A\equiv0$), only the intrinsic and normal averaged Ricci terms remain. In a space form of sectional curvature $\kappa$,
\[
\overline{\mathrm{Ric}}_{\Pi}=(d-1)\,\kappa,\qquad
\overline{\mathrm{Ric}}^{\perp}_{\Pi}=\kappa,\qquad
\sum_{i=1}^{d}\sum_{\alpha=1}^{n-d}\!K(e_i,n_\alpha)=d(n-d)\,\kappa,
\]
so \eqref{eq:c2decomp} reduces to a linear combination of $(d-1)\kappa$ and $\kappa$ (plus $A,H$ if present).

\begin{example}[Space forms]
Let $(M^n,g)$ be a space form with constant sectional curvature $\kappa$. For any $d$-dimensional subspace $\Pi\subset T_pM$,
\[
\overline{\mathrm{Ric}}_{\Pi}
=\frac{2}{d}\!\!\sum_{1\le i<j\le d}\!K(e_i,e_j)
=\frac{2}{d}\cdot\binom{d}{2}\kappa
=(d-1)\kappa,
\]
and
\[
\overline{\mathrm{Ric}}^{\perp}_{\Pi}
=\frac{1}{d(n-d)}\sum_{i=1}^d\sum_{\alpha=1}^{n-d}K(e_i,n_\alpha)
=\frac{1}{d(n-d)}\cdot d(n-d)\,\kappa
=\kappa.
\]
\end{example}

\begin{example}[Heisenberg group]
Let $\mathbb{H}^3$ be the (three-dimensional) Heisenberg group with a left-invariant metric making $\{X,Y,Z\}$ orthonormal and $[X,Y]=Z$ (others zero). The sectional curvatures are
\[
K(X,Y)=-\tfrac{3}{4},\qquad K(X,Z)=K(Y,Z)=\tfrac{1}{4}.
\]
For the $2$-plane $\Pi=\mathrm{span}\{X,Y\}$,
\[
\overline{\mathrm{Ric}}_{\Pi}=K(X,Y)=-\tfrac{3}{4},\qquad
\overline{\mathrm{Ric}}^{\perp}_{\Pi}
=\frac{1}{2}\Big(K(X,Z)+K(Y,Z)\Big)
=\frac{1}{2}\!\left(\tfrac{1}{4}+\tfrac{1}{4}\right)
=\tfrac{1}{4}.
\]
\end{example}

\begin{example}[Complex projective space $\mathbb{CP}^n$ (Fubini--Study), intrinsic/normal means and Bochner bound]
We normalize the Fubini--Study metric so that the \emph{holomorphic sectional curvature} equals $4$.
On a complex space form with this normalization, the curvature tensor is
{\footnotesize
\[
R(X,Y,Z,W)=\langle X,W\rangle\langle Y,Z\rangle-\langle X,Z\rangle\langle Y,W\rangle
+\langle JX,W\rangle\langle JY,Z\rangle-\langle JX,Z\rangle\langle JY,W\rangle
-2\,\langle JX,Y\rangle\langle JZ,W\rangle .
\tag{$\ast$}
\]}
\medskip
\noindent\emph{K\"ahler angle formula.}
Let $\sigma=\mathrm{span}\{u,v\}$ be a real $2$-plane with $u,v$ orthonormal and K\"ahler angle $\varphi$ defined by
$\cos\varphi:=|\langle Ju,v\rangle|$.
Then, using $(\ast)$ with $Z=v$, $W=u$,
\[
\begin{aligned}
K(\sigma)=R(u,v,v,u)
&= \underbrace{\langle u,u\rangle\langle v,v\rangle-\langle u,v\rangle^2}_{=\,1}
+ \underbrace{\langle Ju,u\rangle\langle Jv,v\rangle-\langle Ju,v\rangle\langle Jv,u\rangle}_{=\,-\langle Ju,v\rangle\langle -Jv,u\rangle=\langle Ju,v\rangle^2} \\
&\quad - 2\,\langle Ju,v\rangle \langle Jv,u\rangle
= 1 + \langle Ju,v\rangle^2 - 2\,\big(-\langle Ju,v\rangle^2\big) \\
&= 1 + 3\,\langle Ju,v\rangle^2 \;=\; 1+3\cos^2\varphi.
\end{aligned}
\]
Hence $K(\sigma)\in[1,4]$, with $K=4$ for $J$-invariant (holomorphic) planes and $K=1$ for totally real planes.

\medskip
\noindent\emph{Intrinsic mean on a complex $k$-plane.}
Let $\Pi\subset T_p\mathbb{CP}^n$ be $J$-invariant of real dimension $2k$ ($1\le k\le n$). Choose a $J$-adapted orthonormal basis
$\{e_1,Je_1,\dots,e_k,Je_k\}$ of $\Pi$.
Among the $\binom{2k}{2}$ unordered pairs:
exactly $k$ are holomorphic planes $\mathrm{span}\{e_i,Je_i\}$ (each $K=4$) and the remaining $2k(k-1)$ are totally real (each $K=1$).
Therefore
\[
\sum_{1\le a<b\le 2k} K(e_a,e_b)= 4k + 2k(k-1)=2k(k+1).
\]
Averaging (using $\overline{\mathrm{Ric}}_\Pi=\tfrac{2}{2k}\sum_{a<b}K(e_a,e_b)$) gives
\[
\boxed{\;\overline{\mathrm{Ric}}_\Pi=2(k+1). \;}
\]
In particular, for $k=1$ (complex lines) $\overline{\mathrm{Ric}}_\Pi=4$, and for $k=n$ one gets $\overline{\mathrm{Ric}}_\Pi=2(n+1)=\mathrm{Scal}/(2n)$ (since $\mathrm{Ric}=2(n+1)g$).

\medskip
\noindent\emph{Normal mean for a complex $k$-plane.}
Let $\Pi^\perp$ be the $J$-invariant orthogonal complement (real dimension $2(n-k)$). For $v\in\Pi$ and any $n\in\Pi^\perp$ we have $\langle Jv,n\rangle=0$ (because $Jv\in\Pi$), so every mixed plane $\mathrm{span}\{v,n\}$ is totally real, hence $K(v,n)=1$. Thus
\[
\boxed{\;\overline{\mathrm{Ric}}^{\perp}_\Pi=1. \;}
\]
Consequently, for a simple unit $d$-vector spanning a complex $k$-plane ($d=2k$),
\[
\big\langle \mathcal R_{2k}V,V\big\rangle
= d\,(2n-d)\,\overline{\mathrm{Ric}}^{\perp}_\Pi
= d(2n-d).
\]

\medskip
\noindent\emph{Normal mean for a totally real $d$-plane.}
Let $\Pi$ be totally real of real dimension $d$ ($1\le d\le n$), with orthonormal basis $\{e_1,\dots,e_d\}$, and extend to an orthonormal basis of $\Pi^\perp$ by
\[
\{Je_1,\dots,Je_d\}\ \cup\ \{f_1,Jf_1,\dots,f_{n-d},Jf_{n-d}\},
\]
where $N:=\mathrm{span}\{f_\beta,Jf_\beta\}$ is the $J$-invariant complement of $\Pi\oplus J\Pi$.
For fixed $i$,
\[
K(e_i,Je_i)=4,\qquad
K(e_i,Je_j)=1\ (j\ne i),\qquad
K(e_i,f_\beta)=K(e_i,Jf_\beta)=1.
\]
Summing over all normals and then over $i$,
\[
\sum_{i=1}^d\sum_{\alpha=1}^{2n-d} K(e_i,n_\alpha)
= d\big(4+(d-1)+2(n-d)\big)=d\,(2n-d+3).
\]
Hence
\[
\boxed{\;\overline{\mathrm{Ric}}^{\perp}_\Pi
=\frac{1}{d(2n-d)}\sum_{i,\alpha}K(e_i,n_\alpha)
= 1+\frac{3}{\,2n-d\,}. \;}
\]
For $d=1$ this gives $\overline{\mathrm{Ric}}^{\perp}_\Pi=\tfrac{2(n+1)}{2n-1}=\mathrm{Ric}(u,u)/(2n-1)$, consistent with $\mathrm{Ric}=2(n+1)g$.

\medskip
\noindent\emph{Bochner lower bound for simple eigenfields.}
If $V\not\equiv0$ is a \emph{simple} $d$-vector eigenfield on $(\mathbb{CP}^n,g_{FS})$,
$\Delta_H V=\lambda V$, and $V$ spans a complex $k$-plane ($d=2k$), then by $\overline{\mathrm{Ric}}^{\perp}_\Pi=1$ and Proposition~\ref{prop:Rd-simple},
\[
\boxed{\;\lambda \;\ge\; d(2n-d). \;}
\]
In particular, every harmonic simple $d$-vector tangent to a complex $k$-plane vanishes unless $d\in\{0,2n\}$.
\end{example}

\begin{example}[Surfaces of revolution in $\mathbb{R}^{3}$]
Let $S$ be a surface of revolution parametrized by
\[
x(u,v)=(r(u)\cos v,\; r(u)\sin v,\; z(u)).
\]
At $(u_0,v_0)$ the principal curvatures are
\[
\kappa_1=\frac{z''(u_0)}{(1+z'(u_0)^2)^{3/2}},\qquad
\kappa_2=\frac{z'(u_0)}{r(u_0)\,(1+z'(u_0)^2)^{1/2}},
\]
so the Gaussian curvature is $K=\kappa_1\kappa_2$. Since $d=2$, the intrinsic mean equals the sectional curvature:
\[
\overline{\mathrm{Ric}}_{\Pi}=K\quad\text{for }\Pi=T_pS.
\]
Examples:
\begin{itemize}
\item Cylinder: $r(u)=R$, $z(u)=u$ $\Rightarrow$ $K=0$.
\item Sphere of radius $R$: $r(u)=R\sin(u/R)$, $z(u)=R\cos(u/R)$ $\Rightarrow$ $K=1/R^{2}$.
\item Catenoid: $r(u)=\cosh u$, $z(u)=u$ $\Rightarrow$ $K=-1/\cosh^{4}u$.
\end{itemize}
\end{example}

\begin{example}[Products of space forms $S^a(\rho_a)\times S^b(\rho_b)$]
Let $M=S^a(\rho_a)\times S^b(\rho_b)$ with the product metric, $a,b\ge1$, $n=a+b$, and set
\[
\kappa_a=\rho_a^{-2},\qquad \kappa_b=\rho_b^{-2},
\]
so that each factor has constant sectional curvature $\kappa_a$ and $\kappa_b$, respectively. In a Riemannian product:
\begin{itemize}
\item sectional curvatures of $2$-planes \emph{tangent to a single factor} equal the curvature of that factor, and
\item \emph{mixed} $2$-planes (one vector in each factor) have sectional curvature $0$.
\end{itemize}
Fix $p=(p_A,p_B)\in M$ and decompose any $d$-plane $\Pi\subset T_pM$ as an orthogonal sum
\[
\Pi=\Pi_A\oplus \Pi_B,\qquad \Pi_A\subset T_{p_A}S^a(\rho_a),\ \Pi_B\subset T_{p_B}S^b(\rho_b),
\]
with $d_1:=\dim\Pi_A$, $d_2:=\dim\Pi_B$, and $d=d_1+d_2$. Then:
\begin{align*}
\overline{\mathrm{Ric}}_{\Pi}
&=\frac{2}{d}\Bigg(\binom{d_1}{2}\kappa_a+\binom{d_2}{2}\kappa_b\Bigg)
=\frac{d_1(d_1-1)\,\kappa_a+d_2(d_2-1)\,\kappa_b}{d},\\[4pt]
\overline{\mathrm{Ric}}^{\perp}_{\Pi}
&=\frac{1}{d(n-d)}\Big(d_1(a-d_1)\,\kappa_a+d_2(b-d_2)\,\kappa_b\Big).
\end{align*}
\emph{Derivation.} Mixed sectional curvatures vanish, so only $2$-planes lying in a single factor contribute. Inside $\Pi$, the number of $2$-planes in $\Pi_A$ (resp.\ $\Pi_B$) is $\binom{d_1}{2}$ (resp.\ $\binom{d_2}{2}$), yielding the intrinsic mean. For the normal mean, each $X\in\Pi_A$ sees $(a-d_1)$ normal directions in $T_{p_A}S^a(\rho_a)$ contributing $\kappa_a$, and $b$ directions in $T_{p_B}S^b(\rho_b)$ contributing $0$; symmetrically for $Y\in\Pi_B$.

\medskip
\noindent\emph{Useful special cases.}
\begin{itemize}
\item If $\Pi\subset T_{p_A}S^a(\rho_a)$ with $\dim\Pi=d$ (\(d_1=d\), \(d_2=0\)):
\[
\overline{\mathrm{Ric}}_{\Pi}= (d-1)\kappa_a,\qquad
\overline{\mathrm{Ric}}^{\perp}_{\Pi}=\frac{a-d}{\,n-d\,}\,\kappa_a.
\]
\item If $\Pi\subset T_{p_B}S^b(\rho_b)$ with $\dim\Pi=d$ (\(d_1=0\), \(d_2=d\)):
\[
\overline{\mathrm{Ric}}_{\Pi}= (d-1)\kappa_b,\qquad
\overline{\mathrm{Ric}}^{\perp}_{\Pi}=\frac{b-d}{\,n-d\,}\,\kappa_b.
\]
\item If $\Pi=\mathrm{span}\{X,Y\}$ with $X\in T_{p_A}S^a(\rho_a)$, $Y\in T_{p_B}S^b(\rho_b)$ (\(d_1=d_2=1\), \(d=2\)):
\[
\overline{\mathrm{Ric}}_{\Pi}=0,\qquad
\overline{\mathrm{Ric}}^{\perp}_{\Pi}=\frac{(a-1)\,\kappa_a+(b-1)\,\kappa_b}{2(n-2)}.
\]
\end{itemize}
These formulas plug directly into the Bochner curvature term $d(n-d)\,\overline{\mathrm{Ric}}^{\perp}_{\Pi}$ for simple $d$-vectors $V$ spanning $\Pi$.
\end{example}

\begin{example}[Warped products: intrinsic and normal means]\label{ex:warped-means}
Let $M=B^b\times_f F^m$ with metric $g=g_B+(f\circ\pi_B)^2 g_F$, $n=b+m$, and $f:B\to\mathbb{R}_{>0}$.
All vectors below are \emph{unit in $(M,g)$}. Gradients/Hessians of $f$ are taken on $(B,g_B)$.
At $p=(p_B,p_F)$ the sectional curvatures (Bishop--O'Neill; see \cite{BishopONeill1969,oneill1983semi}) satisfy, for $X,Y\in T_{p_B}B$ and $U,V\in T_{p_F}F$,
\[
K(X,Y)=K_B(X,Y),\qquad
K(U,V)=\frac{1}{f^2}\!\left(K_F(U,V)-|\nabla f|^2\right),\qquad
K(X,U)=-\frac{\mathrm{Hess} f(X,X)}{f}.
\]
Let $\Pi\subset T_pM$ decompose orthogonally as $\Pi=\Pi_B\oplus\Pi_F$ with $\dim\Pi_B=d_1$, $\dim\Pi_F=d_2$, and $d=d_1+d_2$.
Choose orthonormal frames $\{e_i\}_{i=1}^{d_1}\subset\Pi_B$, $\{u_\alpha\}_{\alpha=1}^{d_2}\subset\Pi_F$, and complete to orthonormal frames
$\{b_\beta\}_{\beta=1}^{b-d_1}\subset \Pi_B^\perp\cap T_{p_B}B$ and $\{w_\gamma\}_{\gamma=1}^{m-d_2}\subset \Pi_F^\perp\cap T_{p_F}F$.

\medskip
\noindent\emph{(A) Pure base: $\Pi\subset T_{p_B}B$ ($d_1=d$, $d_2=0$).}
\[
\overline{\mathrm{Ric}}_{\Pi}
=\frac{2}{d(d-1)}\sum_{1\le i<j\le d}K_B(e_i,e_j),\qquad
\overline{\mathrm{Ric}}^{\perp}_{\Pi}
=\frac{1}{d(n-d)}\!\left[\sum_{i=1}^{d}\sum_{\beta=1}^{b-d}\!K_B(e_i,b_\beta)\;-\;\frac{m}{f}\,\mathrm{tr}_{\Pi}(\mathrm{Hess} f)\right].
\]

\noindent\emph{(B) Pure fiber: $\Pi\subset T_{p_F}F$ ($d_1=0$, $d_2=d$).}
\[
\overline{\mathrm{Ric}}_{\Pi}
=\frac{2}{d(d-1)}\sum_{\alpha<\beta}\frac{1}{f^2}\Big(K_F(u_\alpha,u_\beta)-|\nabla f|^2\Big)
=\frac{1}{f^2}\,\overline{\mathrm{Ric}}^{F}_{\Pi}\;-\;\frac{|\nabla f|^2}{f^2},
\]
\[
\overline{\mathrm{Ric}}^{\perp}_{\Pi}
=\frac{1}{d(n-d)}\!\left[\frac{1}{f^2}\sum_{\alpha=1}^{d}\sum_{\gamma=1}^{m-d}K_F(u_\alpha,w_\gamma)\;-\;\frac{d(m-d)}{f^2}|\nabla f|^2\;-\;\frac{d}{f}\,\mathrm{tr}_{T_{p_B}B}(\mathrm{Hess} f)\right].
\]

\noindent\emph{(C) Mixed $2$–planes: $d=2$ with $\Pi=\mathrm{span}\{X,U\}$, $X\in T_{p_B}B$, $U\in T_{p_F}F$.}
Since $d=2$, $\overline{\mathrm{Ric}}_{\Pi}=K(\Pi)$, hence
\[
\overline{\mathrm{Ric}}_{\Pi}=K(X,U)=-\frac{\mathrm{Hess} f(X,X)}{f}.
\]
Writing $n-d=b+m-2$ and completing $X$ in $B$ to $\{X,b_\beta\}$ and $U$ in $F$ to $\{U,w_\gamma\}$,
{\small
\[
\overline{\mathrm{Ric}}^{\perp}_{\Pi}
=\frac{1}{2(b+m-2)}\!\left[\sum_{\beta=1}^{\,b-1}\!K_B(X,b_\beta)
+\frac{1}{f^2}\sum_{\gamma=1}^{\,m-1}\!\Big(K_F(U,w_\gamma)-|\nabla f|^2\Big)
-\frac{m-1}{f}\mathrm{Hess} f(X,X)-\frac{1}{f}\sum_{\beta=1}^{\,b-1}\mathrm{Hess} f(b_\beta,b_\beta)\right].
\]}

\noindent\emph{(D) Mixed $d>2$–planes: general split $\Pi=\Pi_B\oplus\Pi_F$ with $d_1,d_2\ge1$, $d=d_1+d_2$.}
\[
\overline{\mathrm{Ric}}_{\Pi}
=\frac{2}{d(d-1)}\!\left[
\sum_{1\le i<j\le d_1}K_B(e_i,e_j)
+\frac{1}{f^2}\!\sum_{1\le \alpha<\beta\le d_2}\!\Big(K_F(u_\alpha,u_\beta)-|\nabla f|^2\Big)
-\frac{d_2}{f}\,\mathrm{tr}_{\Pi_B}(\mathrm{Hess} f)
\right].
\]

{\footnotesize
\[
\overline{\mathrm{Ric}}^{\perp}_{\Pi}
=\frac{1}{d(n-d)}\!\left[
\sum_{i=1}^{d_1}\sum_{\beta=1}^{b-d_1}\!K_B(e_i,b_\beta)
+\frac{1}{f^2}\sum_{\alpha=1}^{d_2}\sum_{\gamma=1}^{m-d_2}\! \Big(K_F(u_\alpha,w_\gamma)-|\nabla f|^2\Big)
-\frac{m-d_2}{f}\,\mathrm{tr}_{\Pi_B}(\mathrm{Hess} f)
-\frac{d_2}{f}\,\mathrm{tr}_{\Pi_B^\perp}(\mathrm{Hess} f)
\right],
\]}
where $\mathrm{tr}_{\Pi_B}(\mathrm{Hess} f)=\sum_{i=1}^{d_1}\mathrm{Hess} f(e_i,e_i)$ and
$\mathrm{tr}_{\Pi_B^\perp}(\mathrm{Hess} f)=\sum_{\beta=1}^{b-d_1}\mathrm{Hess} f(b_\beta,b_\beta)$.
\end{example}

\section*{Appendix A. Proof of Proposition~\ref{prop:Rd-simple}}

We use the sign convention
\[
R(X,Y)Z=\nabla_X\nabla_YZ-\nabla_Y\nabla_XZ-\nabla_{[X,Y]}Z,
\]
so that the sectional curvature of the $2$-plane spanned by orthonormal $e_i,e_j$ is
$K(e_i,e_j)=R_{ijij}$.

  Let $\{e_1,\dots,e_n\}$ be a local orthonormal frame with $e_i=X_i$ for
  $1\le i\le d$ and $e_{d+\alpha}=n_\alpha$ for $1\le\alpha\le n-d$. Let
  $\{e^1,\dots,e^n\}$ be the dual coframe. It is convenient to identify $V$ with
  the covariant $d$-form $\omega=e^1\wedge\cdots\wedge e^d$, since the
  Weitzenb\"ock operator is classically written on forms.

  For a $d$-form $\omega$, the curvature endomorphism in the Weitzenb\"ock formula is
  \begin{equation}\label{eq:Weitzenbock-p}
    (\mathcal R_d\omega)_{i_1\cdots i_d}
    =\sum_{s=1}^d \mathrm{Ric}_{i_s}{}^{j}\,\omega_{i_1\cdots j\cdots i_d}
    -\sum_{1\le s<t\le d} R_{i_si_t}{}^{jk}\,\omega_{jk\,i_1\cdots\widehat{i_s}\cdots\widehat{i_t}\cdots i_d}.
  \end{equation}
  Because the wedge basis is orthonormal, $\langle\mathcal R_d\omega,\omega\rangle$ is the sum of the
  coefficients that keep the multi-index $\{i_1,\dots,i_d\}$ unchanged.

  Apply \eqref{eq:Weitzenbock-p} to $\omega=e^1\wedge\cdots\wedge e^d$.
  \medskip

  \noindent\emph{(a) Ricci part.}
  In the first sum, only $j=i_s$ contributes (otherwise the multi-index changes), so
  \[
    \sum_{s=1}^d \mathrm{Ric}_{i_s}{}^{j}\,\omega_{i_1\cdots j\cdots i_d}\;\cdot\;\omega_{i_1\cdots i_d}
    \;=\;\sum_{s=1}^d \mathrm{Ric}_{i_s i_s}
    \;=\;\sum_{i=1}^d \mathrm{Ric}(X_i,X_i).
  \]

  \noindent\emph{(b) Riemann part.}
  In the second sum, the only contributions that preserve the multi-index are when
  $\{j,k\}=\{i_s,i_t\}$. The ordered pairs $(j,k)=(i_s,i_t)$ and $(j,k)=(i_t,i_s)$
  both occur and contribute with opposite signs in $\omega_{jk\cdots}$ but also with
  $R_{i_si_t i_t i_s}=-R_{i_si_t i_si_t}$, producing a factor of~$2$:
  \[
    -\sum_{1\le s<t\le d} R_{i_si_t}{}^{jk}\,\omega_{jk\cdots}\;\cdot\;\omega_{i_1\cdots i_d}
    \;=\;-2\sum_{1\le s<t\le d} R_{i_si_t i_si_t}
    \;=\;-2\sum_{1\le s<t\le d} K(e_{i_s},e_{i_t}).
  \]

  Combining (a) and (b),
  \begin{equation}\label{eq:Rd-expanded}
    \big\langle \mathcal R_d\omega,\omega\big\rangle
    =\sum_{i=1}^d \mathrm{Ric}(e_i,e_i)\;-\;2\sum_{1\le i<j\le d} K(e_i,e_j).
  \end{equation}
  Finally, expand $\mathrm{Ric}(e_i,e_i)=\sum_{j\ne i}K(e_i,e_j)$ and split the sum over $j$
  into $j\in\{1,\dots,d\}$ and $j\in\{d+1,\dots,n\}$:
  \[
    \sum_{i=1}^d \mathrm{Ric}(e_i,e_i)
    =2\sum_{1\le i<j\le d}K(e_i,e_j)\;+\;\sum_{i=1}^d\sum_{\alpha=1}^{n-d} K(e_i,n_\alpha).
  \]
  Subtracting the $2\sum_{i<j}K(e_i,e_j)$ term in \eqref{eq:Rd-expanded} leaves exactly the mixed sum,
  which is \(\sum_{i,\alpha}K(X_i,n_\alpha)\). This proves \eqref{eq:Rd-on-simple} for forms; via the musical
  isomorphism the same identity holds for $V\in\Lambda^dT_pM$.
  The Bochner identity \eqref{eq:bochner-simple-d} is the standard
  \(\frac12\Delta|V|^2=\langle \Delta_{H} V,V\rangle-|\nabla V|^2\) with
  \(\Delta_{H}=\nabla^*\nabla+\mathcal R_d\).

\begin{remark}
  For $d=2$ the formula reduces to
  $\langle \mathcal R_2(X\wedge Y),X\wedge Y\rangle=\mathrm{Ric}(X,X)+\mathrm{Ric}(Y,Y)-2K(X,Y)
  =2(n-2)\,\overline{\mathrm{Ric}}^{\perp}_{\Pi}$.
\end{remark}

\bibliographystyle{plain}
\bibliography{no9_paper}

\end{document}